\documentclass[12pt]{amsart}
\usepackage{cite}
\usepackage{epic}
\usepackage[pagebackref]{hyperref}
\usepackage{hyperref}

\newtheorem{theorem}{Theorem}[section]
\newtheorem{lemma}[theorem]{Lemma}
\newtheorem{corollary}[theorem]{Corollary}
\newtheorem{conjecture}[theorem]{Conjecture}
\theoremstyle{definition}
\newtheorem{Def}[theorem]{Definition}
\theoremstyle{remark}
\newtheorem{remark}[theorem]{Remark}
\numberwithin{equation}{section}

\DeclareMathOperator\ai{ai}
\DeclareMathOperator\asc{asc}
\DeclareMathOperator\da{da}
\DeclareMathOperator\dd{dd}
\DeclareMathOperator\des{des}
\DeclareMathOperator\exc{exc}
\DeclareMathOperator\fix{fix}
\DeclareMathOperator\inv{inv}
\DeclareMathOperator\maj{maj}
\DeclareMathOperator\peak{peak}
\DeclareMathOperator\valley{valley}
\DeclareMathOperator\cpeak{cpeak}
\DeclareMathOperator\cvalley{cval}
\DeclareMathOperator\cdfall{cdfall}
\DeclareMathOperator\cdrise{cdrise}
\DeclareMathOperator\fdes{fdes}
\DeclareMathOperator\fmax{fmax}
\DeclareMathOperator\drop{drop}
\DeclareMathOperator\cros{cros}
\DeclareMathOperator\nest{nest}

\newcommand{\PRW}{\mathrm{PRW}}
\renewcommand{\S}{\mathfrak{S}}

\newcommand{\Z}{\mathbb{Z}}

\usepackage{xcolor}

\begin{document}

\title[Around the $q$-binomial-Eulerian polynomials]
{Around the $q$-binomial-Eulerian polynomials}

\author[Z. Lin]{Zhicong Lin}
\address[Zhicong Lin]{School of Science, Jimei University, Xiamen 361021, P.R. China}
\email{zhicong.lin@univie.ac.at}

\author[D.G.L. Wang]{David G.L. Wang$^\dag$$^\ddag$}
\address{
$^\dag$School of Mathematics and Statistics, Beijing Institute of 
Technology, Beijing 102488, P.R.~China\\
$^\ddag$Beijing Key Laboratory on MCAACI, Beijing Institute of 
Technology, Beijing 102488, P.R.~China}
\email{david.combin@gmail.com}

\author[J. Zeng]{Jiang Zeng}
\address[Jiang Zeng]{Univ Lyon, Universit\'e Claude Bernard Lyon 1, CNRS UMR 5208, Institut Camille Jordan, F-69622 Villeurbanne, France}
\email{zeng@math.univ-lyon1.fr}

\begin{abstract}
We find a combinatorial interpretation of Shareshian and Wachs' $q$-binomial-Eulerian polynomials, which leads to an alternative proof of their $q$-$\gamma$-positivity using group actions. Motivated by the sign-balance identity of D\'esarm\'enien--Foata--Loday for the $(\mathrm{des}, \mathrm{inv})$-Eulerian polynomials, we further investigate the sign-balance of the $q$-binomial-Eulerian polynomials. We show the unimodality of the resulting signed binomial-Eulerian polynomials by exploiting their continued fraction expansion and making use of a new quadratic recursion for the $q$-binomial-Eulerian polynomials. We finally use the method of continued fractions to derive a new $(p,q)$-extension of the $\gamma$-positivity of binomial-Eulerian polynomials which involves crossings and nestings of permutations.
\end{abstract}


\maketitle

\section{Introduction}
Let $\S_n$ be the set of all permutations of $[n]:=\{1,2,\ldots,n\}$. For any permutation $\pi=\pi_1\pi_2\cdots\pi_n\in\S_n$, 
the number of \emph{descents}, the number of \emph{excedances}, the \emph{inversion number} and the \emph{major index} of $\pi$ are defined,  respectively, by
\begin{align*}
\des(\pi)&:=|\{i\in[n-1]\colon \pi_i>\pi_{i+1}\}|,\\
\exc(\pi)&:=|\{i\in[n-1]\colon\pi_i>i\}|,\\
\inv(\pi)&:=|\{(i,j)\in[n]\times[n]\colon i<j\text{ and }\pi_i>\pi_j\}|,\\
\maj(\pi)&:=\sum_{\pi_i>\pi_{i+1}} i.
\end{align*}
The first two statistics are  {\em Eulerian statistics} whose
 enumerative polynomials give the $n$th {\em Eulerian polynomial} (cf.~\cite[Sec.~1.3]{st0}) 
$$
A_n(t)=\sum_{\pi\in\S_n}t^{\des(\pi)}=\sum_{\pi\in\S_n}t^{\exc(\pi)},
$$
while the other two statistics are {\em Mahonian statistics} 
with common generating function
\[
[n]_q!:=\prod_{i=1}^n(1+q+\dots+q^{i-1}).
\] 
The joint distributions of Eulerian and Mahonian statistics on permutations have been widely studied; 
see~\cite{Ath,bu,fh,lin,lsw,sw0,sw1,st,sz}.

The {\em$(\maj,\exc)$-Eulerian polynomials} $A_n(t,q)$, which arise in Shareshian and Wachs' study of poset topology~\cite{sw0}, are defined as 
$$
A_n(t,q)=\sum_{\pi\in\S_n} t^{\exc(\pi)}q^{\maj(\pi)-\exc(\pi)}.
$$
Their exponential generating function has a nice $q$-analog 
of Euler's formula (see~\cite{sw1,fh}),
\begin{equation}\label{exc:maj}
\sum_{n\geq0}A_n(t,q)\frac{z^n}{(q;q)_n}=\frac{(1-t)e(z;q)}{e(tz;q)-te(z;q)},
\end{equation}
 where 
\[
(q;q)_n :=\prod_{i=1}^{n}(1-q^i)
\quad\text{and}\quad
e(z;q):=\sum_{n\geq 0}\frac{z^n}{(q;q)_n}.
\]
An \emph{admissible inversion} of a permutation 
 $\pi\in\S_n$ is an inversion pair $(\pi_i,\pi_j)$
satisfying either of the following conditions:
 \begin{itemize}
 \item $1<i$ and $\pi_{i-1}<\pi_i$ or 
 \item there is some $k$ such that $i<k<j$ and $\pi_i<\pi_k$.
 \end{itemize}
Let $\ai(\pi)$ be the number of admissible inversions of $\pi$. 
For example, the admissible inversions of $3142$ are $(3,2)$ and $(4,2)$. 
So $\ai(\pi)=2$. 
The statistic of admissible inversions was first introduced by Shareshian and Wachs~\cite{sw0}, 
who gave the interpretation
\begin{equation}\label{q-eul}
A_n(t,q)=\sum_{\pi\in\S_n}t^{\des(\pi)}q^{\ai(\pi)}.
\end{equation}
The detailed proof of this interpretation was given by Linusson, Shareshian and Wachs~\cite{lsw} using Rees products of posets; see~\cite{bu,lin} for alternative approaches and a generalization.
 
It is known (cf.~\cite{Petersen2015}) that 
the Eulerian polynomials are 
 the $h$-polynomials of dual permutohedra.
Postnikov, Reiner, and Williams~\cite[Section~10.4]{prw} proved that 
the $h$-polynomials of dual stellohedra
equal the binomial transformations 
$$
\tilde{A}_n(t)=1+t\sum_{m=1}^n{n\choose m}A_m(t)
$$
of the Eulerian polynomials,
and provided the combinatorial interpretation
\begin{equation}\label{prw}
\tilde{A}_n(t)=\sum_{\pi\in\PRW_{n+1}}t^{\des(\pi)},
\end{equation}
where $\PRW_n$ is the set of permutations $\pi\in\S_n$ such that the first ascent of $\pi$ appears at the letter $1$ if $\pi$ has an ascent.
For example, 
\[
\PRW_1=\{1\},\quad
\PRW_2=\{12,21\},\quad\text{and}\quad
\PRW_3=\{123,132,213,312,321\}.
\]
Shareshian and Wachs~\cite{sw2} called $\tilde{A}_n(t)$ {\em binomial-Eulerian polynomials} 
and introduced the $q$-{\em binomial-Eulerian polynomials} 
$$
\tilde{A}_n(t,q)=1+t\sum_{m=1}^n{n\brack m}_qA_m(t,q),
$$
where 
\[
{n\brack m}_q=\frac{(q;q)_n}{(q;q)_m (q;q)_{n-m}}
\]
are the $q$-binomial coefficients.

Even though an algebro-geometric interpretation of $\tilde{A}_n(t,q)$ has already been found  in~\cite{sw2}, 
no combinatorial interpretation of $\tilde{A}_n(t,q)$ is known
similar to classical Eulerian polynomials. Our first aim is to 
give such an interpretation, which 
is a $q$-analog of \eqref{prw} and is similar to the interpretation~\eqref{q-eul} for $A_n(t,q)$.
\begin{theorem}\label{int:bin}
For $n\geq1$, the $q$-binomial-Eulerian polynomial $\tilde{A}_n(t,q)$ has the interpretation
\begin{equation*}
\tilde{A}_n(t,q)=\sum_{\pi\in\PRW_{n+1}}t^{\des(\pi)}q^{\ai(\pi)}.
\end{equation*}
\end{theorem}

Recall that a polynomial $\sum_{i=0}^nh_it^i$ in $t$ with real coefficients is said to be {\em palindromic}
if $h_i=h_{n-i}$ for all $0\leq i\leq \lfloor n/2\rfloor$. It is {\em unimodal} if 
$$h_0\leq h_1\leq\cdots\leq  h_c\geq h_{c+1}\geq\cdots\geq h_n\qquad\text{for some $c$.}$$
A stronger property implying both the palindromicity and the unimodality is the $\gamma$-positivity. 
A polynomial of degree $n$ in $t$ with real coefficients is said to be {\em$\gamma$-positive} if it can be written in the basis
$$
\{t^k(t+1)^{n-2k}\}_{0\leq k\leq n/2}
$$
with non-negative coefficients. Many interesting polynomials arising in enumerative and geometric combinatorics are palindromic and unimodal, some of which are even $\gamma$-positive; see~\cite{Ath,br2,Petersen2015}. 

For a permutation $\sigma=\sigma_1\cdots\sigma_n\in\S_n$, we call $\sigma_i$ ($1\leq i\leq n$) a {\em double descent} (resp.~{\em double ascent, peak, valley}) of $\sigma$ if $\sigma_{i-1}>\sigma_i>\sigma_{i+1}$ (resp.~$\sigma_{i-1}<\sigma_i<\sigma_{i+1}$, $\sigma_{i-1}<\sigma_i>\sigma_{i+1}$, $\sigma_{i-1}>\sigma_i<\sigma_{i+1}$), where we use the convention $\sigma_0=\sigma_{n+1}=+\infty$. In particular, $\sigma_1$ is a double descent if $\sigma_1>\sigma_2$,  and in this case we will call $\sigma_1$ the {\em initial double descent}. Denote by $\dd(\sigma)$ (resp.~$\da(\sigma)$, $\peak(\sigma)$, $\valley(\sigma)$) the number of non-initial double descents (resp.~double ascents, peaks, valleys) of $\sigma$. 
Foata and  Sch\"uzenberger~\cite[Theorem~5.6]{fsc0} proved the following elegant $\gamma$-positivity expansion of the Eulerian polynomials
\begin{equation}\label{gam:eul}
A_n(t)=\sum_{k=0}^{\lfloor \frac{n-1}{2}\rfloor}\gamma_{n,k}t^k(1+t)^{n-1-2k},
\end{equation}
where $\gamma_{n,k}$ is the cardinality of the set
\begin{equation*}\label{gam:int}
\Gamma_{n,k}:=\{\sigma\in\S_n\colon\dd(\sigma)=0,\ \sigma_1<\sigma_2\text{ and } \des(\sigma)=k\}.
\end{equation*}
The $\gamma$-positivity formula of Postnikov, Reiner, and Williams~\cite[Theorem~11.6]{prw} in the case of stellohedron asserts that
\begin{equation}\label{gam:bino-eul}
\tilde{A}_n(t)=\sum_{k=0}^{\lfloor \frac{n-1}{2}\rfloor}\tilde{\gamma}_{n,k}t^k(1+t)^{n-1-2k},
\end{equation}
where 
$\tilde{\gamma}_{n,k}$ counts permutations $\sigma\in\PRW_{n+1}$ such that $\sigma$ has no double ascents and $\asc(\sigma)=k$, where $\asc(\sigma):=n-1-\des(\sigma)$.

The following $q$-analog of~\eqref{gam:eul} was proved  by various methods in~\cite{sw1,lsw,lz,sw2}: 
$$
A_n(t,q)=\sum_{k=0}^{\lfloor \frac{n-1}{2}\rfloor}\gamma_{n,k}(q)t^k(1+t)^{n-1-2k},
$$
where $\gamma_{n,k}(q)=\sum_{\sigma\in\Gamma_{n,k}}q^{\inv(\sigma)}$,
 and a similar $q$-$\gamma$-positivity
  expansion for $\tilde{A}_n(t,q)$  was recently established by Shareshian and Wachs~\cite[Theorem~4.5]{sw2}.
  
  \begin{theorem}[Shareshian and Wachs]
\label{q-bino:gam}
Let 
$$\tilde{\Gamma}_{n,k}:=\{\sigma\in\S_n\colon \dd(\sigma)=0, \des(\sigma)=k\}.$$
The $q$-binomial-Eulerian polynomials have the $q$-$\gamma$-positivity expansion 
\begin{equation}\label{bino:gam}
\tilde{A}_{n}(t,q)=\sum_{k=0}^{\lfloor\frac{n}{2}\rfloor}\tilde{\gamma}_{n,k}(q)t^k(1+t)^{n-2k},
\end{equation}
where 
\begin{equation*}\label{bino:int}
\tilde{\gamma}_{n,k}(q)=\sum_{\pi\in\tilde{\Gamma}_{n,k}}q^{\inv(\pi)}.
\end{equation*}
\end{theorem}
 Note that the combinatorial meanings of $\tilde{\gamma}_{n,k}(1)$ in~\eqref{bino:gam} and $\tilde{\gamma}_{n,k}$ in~\eqref{gam:bino-eul} are apparently different.   
As observed  in~\cite{sw2}, the existence of  expansion~\eqref{bino:gam} with $\tilde{\gamma}_{n,k}(q)\in\Z[q]$ for $\tilde{A}_n(t,q)$ is equivalent to a symmetric $q$-Eulerian identity  due  independently to  Chung--Graham~\cite{cg} and Han--Lin--Zeng~\cite{hlz}. Theorem~\ref{q-bino:gam} was obtained from the principle specialization of an analogous symmetric function identity in~\cite{sw2}. 
Theorem~\ref{int:bin} together with the so-called {\em Modified Foata--Strehl group action} on permutations enables us to give a combinatorial proof to Theorem~\ref{q-bino:gam}. Our alternative approach has the advantage that 
makes the interpretation of $\tilde{\gamma}_{n,k}$ in~\eqref{gam:bino-eul} transparent; see Remark~\ref{gam:prw}.

 In 1992, D\'esarm\'enien and Foata~\cite{df} showed the following sign-balance identity, which was conjectured by Loday~\cite{lo},
\begin{equation}\label{foata}
\sum_{\pi\in\S_n}t^{\des(\pi)}(-1)^{\inv(\pi)}=
\begin{cases}
\,\,(1-t)^mA_{m}(t), \qquad\qquad&\text{if $n=2m$};\\
\,\,(1-t)^mA_{m+1}(t),&\text{if $n=2m+1$}.
\end{cases}
\end{equation}
This paper stems from the observation that identity~\eqref{foata} follows from  a simple quadratic  recursion~\eqref{rec:desinv} for the $(\inv,\des)$-$q$-Eulerian polynomials. This idea enables us to prove similar sign-balance identities for $A_n(t,q)$ and $\tilde{A}_n(t,q)$.
It appears that the signed binomial-Eulerian polynomials $\tilde{A}_n(t,-1)$ have interesting properties 
which are observable from their first terms:    
\begin{align*}
\tilde{A}_1(t,-1)&=1+t,\\
\tilde{A}_2(t,-1)&=1+t+t^2,\\
\tilde{A}_3(t,-1)&=1+3t+3t^2+t^3,\\
\tilde{A}_4(t,-1)&=1+3t+5t^2+3t^3+t^4,\\
\tilde{A}_5(t,-1)&=1+7t+15t^2+15t^3+7t^4+t^5,\\
\tilde{A}_6(t,-1)&=1+7t+19t^2+25t^3+19t^4+7t^5+t^6.
\end{align*}
Here is the central result of this paper.
\begin{theorem}\label{thm:bino}
For any $n\geq1$, the signed binomial-Eulerian polynomial $\tilde{A}_n(t,-1)$ is palindromic and unimodal. 
\end{theorem}

 Although the palindromicity of $\tilde{A}_n(t,-1)$ 
 follows directly from the $q$-$\gamma$-positivity expansion~\eqref{bino:gam} of $\tilde{A}_n(t,q)$, it is not clear how to derive the unimodality in Theorem~\ref{thm:bino}
 from Theorem~\ref{q-bino:gam}. 
In showing the unimodality of $\tilde{A}_n(t,-1)$, we find a  new quadratic recursion for $\tilde{A}_n(t,q)$.
 
 \begin{theorem}\label{rec:q-bino}
The $q$-binomial-Eulerian polynomials satisfy the recurrence relation
\begin{equation*}
\tilde{A}_{n+1}(t,q)=(1+t)\tilde{A}_{n}(t,q)+t\sum_{k=1}^n{n\brack k}_qq^kA_k(t,q)\tilde{A}_{n-k}(t,q)
\end{equation*}
for $n\geq0$ with initial value $\tilde{A}_{0}(t,q)=1$.
\end{theorem}
 
As will be seen,
two specializations of  this recursion together with a continued fraction expansion conclude the desired unimodality of $\tilde{A}_n(t,-1)$ in~Theorem~\ref{thm:bino}. Via the machinery of continued fraction, we will also prove a new $(p,q)$-extension of the $\gamma$-positivity of binomial-Eulerian polynomials. 

The rest of this paper is organized as follows. 
In Section~\ref{sign:qbino}, we show how to derive~\eqref{foata} and the sign-balance identity of
   the binomial-Eulerian polynomials using appropriate 
   quadratic recursions and prove Theorem~\ref{rec:q-bino}.
In Section~\ref{sec:bino}, we show Theorem~\ref{int:bin} and present the Modified Foata--Strehl group action proof of Theorem~\ref{q-bino:gam}. 
In Section~\ref{sec:unimodal}, via the machinery of continued fraction, we prove the  unimodality of $\tilde{A}_n(t,-1)$ and show a $(p,q)$-extension of the $\gamma$-positivity of  binomial-Eulerian polynomials.     
We end this paper with two log-concavity conjectures.

\section{
Quadratic recursions and sign-balance of $q$-binomial-Eulerian polynomials}
\label{sign:qbino}
In this section, we investigate the sign-balance of $q$-binomial-Eulerian polynomials. We begin with a new simple approach to identity~\eqref{foata}.
The following lemma is useful.
\begin{lemma}[cf.~\cite{df}]
\label{lem:lim}
For any integers $m\geq i\geq0$,
\begin{equation*}
\lim_{q \to -1}{2m \brack 2i}_{q}=\lim_{q \to -1}{2m+1 \brack 2i}_{q}=\lim_{q \to -1}{2m+1 \brack 2i+1}_{q}={m\choose i}\quad\text{and}\quad\lim_{q \to -1}{2m\brack 2i+1}_{q}=0.
\end{equation*}
\end{lemma}

Let us define the {\em$(\inv,\des)$-Eulerian polynomials} by
$$
A_n^{\des,\inv}(t,q):=\sum_{\pi\in\S_n}t^{\des(\pi)}q^{\inv(\pi)}.
$$ 
Chow~\cite{ch} gave a combinatorial proof of  the quadratic  recursion 
\begin{equation}\label{rec:desinv}
A_{n+1}^{\des,\inv}(t,q)=(1+tq^n)A_n^{\des,\inv}(t,q)+t\sum_{k=1}^{n-1}{n\brack k}_qq^kA_{n-k}^{\des,\inv}(t,q)A_k^{\des,\inv}(t,q).
\end{equation}
 Taking $q=1$, we obtain 
\begin{equation}\label{rec:euler}
A_{n+1}(t)=(1+t)A_n(t)+t\sum_{k=1}^{n-1}{n\choose k}A_{n-k}(t)A_k(t).
\end{equation}

\begin{proof}[{\bf A new  simple proof of  \eqref{foata}.}]
 We proceed by induction on $n$. Assume that~\eqref{foata} holds  for $n$ up to $2m-1$.  It then follows from recursion~\eqref{rec:desinv} and Lemma~\ref{lem:lim} that 
\begin{align*}
A_{2m}^{\des,\inv}(t,-1)&=(1-t)A_{2m-1}^{\des,\inv}(t,-1)+t\sum_{k=1}^{m-1}{m-1\choose k}A_{2(m-k)-1}^{\des,\inv}(t,-1)A_{2k}^{\des,\inv}(t,-1)\\
&\quad-t\sum_{k=0}^{m-2}{m-1\choose k}A_{2(m-k-1)}^{\des,\inv}(t,-1)A_{2k+1}^{\des,\inv}(t,-1)\\
&=(1-t)A_{2m-1}^{\des,\inv}(t,-1)=(1-t)^mA_m(t)
\end{align*}
and 
\begin{align*}
A_{2m+1}^{\des,\inv}(t,-1)&=(1+t)A_{2m}^{\des,\inv}(t,-1)+t\sum_{k=1}^{m-1}{m\choose k} A_{2(m-k)}^{\des,\inv}(t,-1)A_{2k}^{\des,\inv}(t,-1)\\
&=(1-t)^mA_{m+1}(t),
\end{align*}
where the last equality follows from  the recurrence relation~\eqref{rec:euler}. This completes the proof of  \eqref{foata} by induction. 
\end{proof}

The first author~\cite[Theorem~2]{lin} showed 
that one can derive the following quadratic 
recursion for $A_n(t,q)$, which is 
  a $q$-analog of recursion~\eqref{rec:euler}:
\begin{equation}\label{lin-A}
A_{n+1}(t,q)=(1+t)A_n(t,q)+t\sum_{k=1}^{n-1}{n\brack k}_qq^kA_k(t,q)A_{n-k}(t,q).
\end{equation}
By applying this recursion, the following major-balance identity can be proved through the same approach as \eqref{foata}, the details of which are omitted due to the similarity. 

\begin{theorem}\label{sign:qeul}
For $n\geq1$, we have 
\begin{equation*}
A_n(t,-1)=
\begin{cases}
\,\,(1+t)^mA_{m}(t), \qquad\qquad&\text{if $n=2m$};\\
\,\,(1+t)^mA_{m+1}(t),&\text{if $n=2m+1$}.
\end{cases}
\end{equation*}
\end{theorem}

The above identity for even $n$  appeared in~\cite[Corollary~6.2]{ssw}. 
An immediate consequence of Theorem~\ref{sign:qeul} and Lemma~\ref{lem:lim} is the following signed identity for $\tilde{A}_n(t,q)$. 


\begin{corollary}\label{sign-qbino}
For $n\geq1$, we have
\begin{equation*}
\tilde{A}_n(t,-1)=
\begin{cases}
\,\,1+t\sum_{k=1}^m{m\choose k}(1+t)^kA_k(t), \qquad\qquad&\text{if $n=2m$};\\
\,\,1+t\sum_{k=0}^m{m\choose k}(1+t)^k(A_k(t)+A_{k+1}(t)),&\text{if $n=2m+1$}.
\end{cases}
\end{equation*}
Here we use the convention $A_0(t)=0$. 
\end{corollary}

In the rest of this section, we give the proof of  Theorem~\ref{rec:q-bino}. The {\em Eulerian differential operator} $\delta_z$ 
used below is defined by
$$\delta_z(f(z)):=\frac{f(z)-f(qz)}{z}$$
for any formal power series $f(z)$ over the ring of real polynomials in $q$.
It is not difficult to show for any variable $\alpha$, that
\begin{align}\label{q-dif}
\delta_z(e(\alpha z;q))=\alpha e(\alpha z;q). 
\end{align}

\begin{proof}[{\bf Proof of Theorem~\ref{rec:q-bino}.}]

We begin with the calculation of  the exponential  generating function of $\tilde{A}_n(t,q)$. 
By using~\eqref{exc:maj}, we can deduce that
\begin{align*}
\sum_{n\geq0}\tilde{A}_n(t,q)\frac{z^n}{(q;q)_n}&=\sum_{n\geq0}\biggl(1+t\sum_{m=1}^n{n\brack m}_qA_m(t,q)\biggr)\frac{z^n}{(q;q)_n}\\
&=\sum_{n\geq0}(1-t)\frac{z^n}{(q;q)_n}+\sum_{n\geq0}\biggl(t\sum_{m=0}^n{n\brack m}_qA_m(t,q)\biggr)\frac{z^n}{(q;q)_n}\\
&=(1-t)e(z;q)+t\left(\sum_{n\geq0}\frac{z^n}{(q;q)_n}\right)\left(\sum_{n\geq0}A_n(t,q)\frac{z^n}{(q;q)_n}\right)\\
&=(1-t)e(z;q)+te(z;q)\frac{(1-t)e(z;q)}{e(tz;q)-te(z;q)},
\end{align*}
which is simplified to 
\begin{equation}\label{bino:euler}
\sum_{n\geq0}\tilde{A}_n(t,q)\frac{z^n}{(q;q)_n}=\frac{(1-t)e(z;q)e(tz;q)}{e(tz;q)-te(z;q)}.
\end{equation}

Applying the operator $\delta_z$ to both sides of ~\eqref{bino:euler} and using property~\eqref{q-dif} and the product rule of the Eulerian differential operator (see~\cite[Lemma~7]{lin}) yields 
\begin{align*}
&\quad\sum_{n\geq0}\tilde{A}_{n+1}(t,r,q)\frac{z^n}{(q;q)_n}\\
&=\delta_z\left(\frac{(1-t)e(z; q)e(tz;q)}{e(tz; q)-te(z;q)}\right)\\
&=\frac{\delta_z((1-t)e(z; q)e(tz;q))}{e(tz; q)-te(z;q)}+\delta_z\left((e(tz; q)-te(z;q))^{-1}\right)(1-t)e(zq; q)e(tzq;q)\\
&=\frac{(1-t)e(tz; q)(te(zq;q)+e(z;q))}{e(tz; q)-te(z;q)}+\frac{(1-t)e(zq; q)e(tzq;q)(te(z; q)-te(tz; q))}{(e(tqz; q)-te(qz;q))(e(tz; q)-te(z;q))}\\
&=\frac{(1-t)e(z;q)e(tz; q)}{e(tz; q)-te(z;q)}+t\frac{(1-t)e(z;q)e(tz; q)}{e(tz; q)-te(z;q)}\frac{(1-t)e(zq;q)}{e(tzq; q)-te(zq;q)}\\
&\quad\ +\frac{t(1-t)e(zq;q)\Delta(t,q)}{(e(tz;q)-te(z;q))(e(tzq; q)-te(zq;q))},
\end{align*}
where  
\begin{align*}
\Delta(t,q)&:=te(tz;q)[e(z;q)-e(zq;q)]-
[e(tz;q)-e(tzq;q)]e(z;q)\\
&=tze(tz;q)\delta_z(e(z;q))-ze(z;q)\delta_z(e(tz;q)).
\end{align*}
Invoking \eqref{q-dif} we see immediately 
that $\Delta(t,q)=0$, and so
$$
\sum_{n\geq0}\tilde{A}_{n+1}(t,r,q)\frac{z^n}{(q;q)_n}=\frac{(1-t)e(z;q)e(tz; q)}{e(tz; q)-te(z;q)}+t\frac{(1-t)e(z;q)e(tz; q)}{e(tz; q)-te(z;q)}\frac{(1-t)e(zq;q)}{e(tzq; q)-te(zq;q)}.\nonumber
$$
Extracting the coefficient of $z^n/(q;q)_n$ from both sides, we obtain Theorem~\ref{rec:q-bino}.
\end{proof}

A direct consequence of Theorem~\ref{rec:q-bino} and Lemma~\ref{lem:lim} is the following recurrence relations for $\tilde{A}_n(t,-1)$, involving the signed Eulerian polynomials $A_n(t,-1)$.
\begin{corollary}\label{cor:sigbino}
For $n\geq0$, we have 
\begin{equation}\label{bino:odd}
\tilde{A}_{2n+1}(t,-1)=(1+t)\tilde{A}_{2n}(t,-1)+t\sum_{k=1}^n{n\choose k}A_{2k}(t,-1)\tilde{A}_{2n-2k}(t,-1)
\end{equation}
and 
\begin{align*}
\tilde{A}_{2n+2}(t,-1)&=(1+t)\tilde{A}_{2n+1}(t,-1)+t\sum_{k=1}^n{n\choose k}A_{2k}(t,-1)\tilde{A}_{2n+1-2k}(t,-1)\\
&\quad-t\sum_{k=0}^n{n\choose k}A_{2k+1}(t,-1)\tilde{A}_{2n-2k}(t,-1).
\end{align*}
\end{corollary}

\section{Proof of Theorems~\ref{int:bin} and \ref{q-bino:gam}}
\label{sec:bino}

We shall prove Theorems~\ref{int:bin} and \ref{q-bino:gam} in Subsections~\ref{ssec1}
and~\ref{ssec2} respectively.

\subsection{A combinatorial interpretation of $\tilde{A}_n(t,q)$}\label{ssec1}

We need the following classical interpretation of the $q$-binomial coefficients (cf.~\cite[Prop.~1.3.17]{st0}):
\begin{equation}\label{eq:qmul}
{n\brack  k}_{q}=\sum_{({\mathcal A},\,{\mathcal B})}q^{\inv({\mathcal A},\,{\mathcal B})},
\end{equation}
where the sum is over all ordered set partitions $({\mathcal A}, {\mathcal B})$ of $[n]$ such that $|{\mathcal A}|=k$ and 
$$
\inv({\mathcal A}, {\mathcal B}):=\{(i,j)\in {\mathcal A}\times {\mathcal B}\colon i>j\}.
$$

\begin{proof}[{\bf Proof of Theorem~\ref{int:bin}}]
We will show that the bivariant polynomial
 $$\tilde{B}_n(t,q):=\sum_{\pi\in\PRW_{n+1}}t^{\des(\pi)}q^{\ai(\pi)}$$
  satisfies the same recurrence relation as $\tilde{A}_n(t,q)$ in Theorem~\ref{rec:q-bino}. For each  $0\leq k\leq n$, let  
$$
\mathcal{B}_{n+1,k}:=\{\pi\in\PRW_{n+1}\colon \pi_{n+1-k}=n+1\}
$$
and introduce the refinement $\tilde{B}_{n,k}(t,q)$ of $\tilde{B}_n(t,q)$ by
$$
\tilde{B}_{n,k}(t,q):=\sum_{\pi\in\mathcal{B}_{n+1,k}}t^{\des(\pi)}q^{\ai(\pi)}.
$$
It is clear that  $\tilde{B}_n(t,q)=\sum_{k=0}^n\tilde{B}_{n,k}(t,q)$, $\tilde{B}_{n,n}(t,q)=t\tilde{B}_{n-1}(t,q)$ and $\tilde{B}_{n,0}(t,q)=\tilde{B}_{n-1}(t,q)$. The desired result then follows from the claim that  
\begin{equation}\label{eq:cla}
\tilde{B}_{n,k}(t,q)=t{n-1\brack  k}_{q}q^kA_k(t,q)\tilde{B}_{n-1-k}(t,q)\quad\text{for any $1\leq k\leq n-1$}. 
\end{equation}

It remains to show the above claim. For a set $X$ of distinct positive integers, we denote by ${X\choose m}$ the $m$-element subsets of $X$, by $\S_X$ the set of permutations of $X$ and by $\PRW_X$ the set of all permutations in $\S_X$ whose first ascent entry is $\min(X)$. Let $\mathcal{W}(n,k)$ be the set of all triples $(W,\pi_L,\pi_R)$ such that $W\in{[n]\setminus\{1\}\choose k}$, $\pi_L\in\PRW_{[n]\setminus W}$ and $\pi_R\in\S_W$. Note that  for every permutation in $\mathcal{B}_{n+1,k}$ ($1\leq k\leq n-1$), the entry $n+1$ appears to the right of the entry $1$. Therefore, one can check easily that the mapping $\pi\mapsto (W,\pi_L,\pi_R)$ defined by 
\begin{itemize}
\item $W=\{\pi_i\colon n+2-k\leq i\leq n+1\}$,
\item $\pi_L=\pi_1\pi_2\cdots\pi_{n-k}$ and $\pi_R=\pi_{n+2-k}\pi_{n+3-k}\cdots\pi_{n+1}$,
\end{itemize}
is a bijection between $\mathcal{B}_{n+1,k}$ and $\mathcal{W}(n,k)$ satisfying 
$$
\des(\pi)=\des(\pi_L)+\des(\pi_R)+1
$$
and 
$$
\ai(\pi)=\ai(\pi_L)+\ai(\pi_R)+\inv([n]\setminus W,W)+k.
$$
It follows from this bijection and the interpretations~\eqref{q-eul} and~\eqref{eq:qmul} that claim~\eqref{eq:cla} holds, which completes the proof. 
\end{proof}

As an example of Theorem~\ref{int:bin}, the permutations in $\PRW_{4}$ with two descents are 
$1432$, $3142$, $4132$, $2143$, $4312$, $4213$ and $3214$,  which contribute the monomial $(2q^2+2q+3)t^2$ to $\tilde{A}_3(t,q)$. 

\subsection{A group-action proof of the $q$-$\gamma$-positivity of $\tilde A_n(t,q)$}\label{ssec2}


Let us review briefly the Modified Foata--Strehl group action originally inspired by work of Foata and Strehl~\cite{fst}. Let $\sigma\in\S_n$, for  any  $x\in[n]$, the {\em$x$-factorization} of $\sigma$ reads
$
\sigma=w_1 w_2x w_3 w_4,
$
 where $w_2$ (resp.~$w_3$) is the maximal contiguous subword immediately to the left (resp.~right)  of $x$ whose letters are all smaller than $x$. Following~\cite{fst} we define $\varphi_x(\sigma)=w_1 w_3x w_2 w_4$.  For instance, if $x=5$ and $\sigma=63157248\in\S_8$, then $w_1=6,w_2=31,w_3=\emptyset$ and $w_4=7248$.
 Thus $\varphi_x(\sigma)=65317248$. Introduce the modified action $\varphi'_x$ on $\sigma$ by
  $$ \varphi'_x(\sigma):=
 \begin{cases}
\varphi_x(\sigma), &\text{if $x$ is a double ascent  or  double descent of $\sigma$};\\
\sigma,& \text{if $x$ is a valley or a peak of $\sigma$.}
\end{cases}
$$

 It is clear that the $\varphi'_x$'s are involutions and  commute. Therefore, for any subset $S\subseteq[n]$ we can define the function $\varphi'_S\colon\S_n\rightarrow\S_n$ by 
$$
\varphi'_S(\sigma)=\prod_{x\in S}\varphi'_x(\sigma),
$$
where the multiplication is the composition of functions.
Hence the group $\mathbb{Z}_2^n$ acts on $\S_n$ via the functions $\varphi'_S$, where $S\subseteq[n]$. This action is called  the {\em Modified Foata--Strehl action} ({\em MFS-action} for short)
 and has a nice visualization as depicted in Fig.~\ref{valhop}. Note that this MFS-action is exactly the same as the  version used in~\cite{lz}.

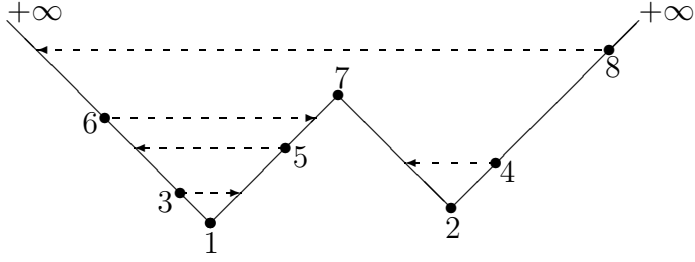
\begin{figure}[h!]
\setlength {\unitlength} {0.8mm}
\begin {picture} (90,40) \setlength {\unitlength} {1mm}
\thinlines
\put(24,8){\dashline{1}(-1,0)(-8,0)}
\put(24,8){\vector(1,0){0.1}}

\put(10,14){\dashline{1}(1,0)(20,0)}
\put(10,14){\vector(-1,0){0.1}}

\put(6,18){\dashline{1}(1,0)(28,0)}
\put(34,18){\vector(1,0){0.1}}

\put(46,12){\dashline{1}(1,0)(12,0)}
\put(46,12){\vector(-1,0){0.1}}

\put(73,27){\dashline{1}(-1,0)(-76,0)}
\put(-3,27){\vector(-1,0){0.1}}

\put(20,4){\line(-1,1){27}}\put(-7,31){$+\infty$}

\put(30,14){\circle*{1.3}}\put(31,11.5){$5$}

\put(6,18){\circle*{1.3}}\put(3,16){$6$}
\put(20,4){\circle*{1.3}}
\put(20,4){\circle*{1.3}}\put(19.1,0){$1$}
\put(16,8){\circle*{1.3}}\put(13,5.5){$3$}
\put(52,6){\circle*{1.3}}\put(51.2,2){$2$}
\put(20,4){\line(1,1){17}}\put(37,21){\circle*{1.3}}
\put(37,21){\line(1,-1){15}}\put(36.5,22){$7$}
\put(58,12){\circle*{1.3}}\put(58.5,9.5){$4$}
\put(52,6){\line(1,1){25}}\put(73,27){\circle*{1.3}}\put(72.5,23.5){$8$}
\put(77,31){$+\infty$}
\end{picture}
\caption{MFS-actions on $63157248$
\label{valhop}}
\end {figure}

\begin{proof}[{\bf Proof of Theorem~\ref{q-bino:gam}}] For any permutation $\sigma\in\PRW_{n+1}$ and $x\in[n+1]$, it is not hard to see that the permutation $\varphi_x(\sigma)$ still has the property that the entry $1$ is the first ascent. Thus, the set $\PRW_{n+1}$ is invariant under the MFS-action. The MFS-action divides the set $\PRW_{n+1}$ into disjoint orbits. Moreover, if $x$ is a double descent (resp.~peak or valley) of $\sigma$, then $x$ is a double ascent (resp.~peak or valley) of the permutation $\varphi'_x(\sigma)$. In the orbit containing $\sigma$, we can choose the unique permutation with least descents (also coincident with the one without double descents), denoted $\bar{\sigma}$, as a representative element. Then, we have $\da(\bar{\sigma})=n-\peak(\bar{\sigma})-\valley(\bar{\sigma})$ and 
   $\des(\bar{\sigma})=\peak(\bar{\sigma})=\valley(\bar{\sigma})-1$.
   
By~\cite[Lemma~7]{lz}, the statistic ``$\ai$'' is constant inside each orbit. Thus, by Theorem \ref{int:bin} and the above discussion, one may deduce that
\begin{align*}
\tilde{A}_n(t,q)=\sum_{\sigma\in\PRW_{n+1}}t^{\des(\sigma)}q^{\ai(\sigma)}&=\sum_{k=0}^{\lfloor n/2\rfloor}\left(\sum_{\bar{\sigma}\in\PRW_{n+1}\cap\Gamma_{n+1,k}}q^{\ai(\bar{\sigma})}\right)t^{k}(1+t)^{n-2k}\\
&=\sum_{k=0}^{\lfloor n/2\rfloor}\left(\sum_{\bar{\sigma}\in\PRW_{n+1}\cap\Gamma_{n+1,k}}q^{\inv(\bar{\sigma})}\right)t^{k}(1+t)^{n-2k}\\
&=\sum_{k=0}^{\lfloor n/2\rfloor}\left(\sum_{\pi\in\tilde{\Gamma}_{n,k}}q^{\inv(\pi)}\right) t^{k}(1+t)^{n-2k},
\end{align*}
where the second last equality is a consequence of~\cite[Lemma~8]{lz}, while the last equality follows from the simple one-to-one correspondence 
$$
\bar{\sigma}_1\bar{\sigma}_2\cdots\bar{\sigma}_n\mapsto (\bar{\sigma}_2-1)\cdots(\bar{\sigma}_n-1)
$$
between $\PRW_{n+1}\cap\Gamma_{n+1,k}$ and $\tilde{\Gamma}_{n,k}$.
Note that the first letter of each $\bar{\sigma}\in\PRW_{n+1}\cap\Gamma_{n+1,k}$ must be $1$. It is easy to check that the above correspondence is a bijection preserving both the number of descents and the number of inversions. This establishes~\eqref{bino:gam}.
\end{proof}

\begin{remark}\label{gam:prw}
In each orbit of the MFS-action on $\PRW_{n+1}$, there is a unique permutation with least ascents, which is exactly the one with no double ascents. Thus, the interpretation of $\tilde{\gamma}_{n,k}$ in~\eqref{gam:bino-eul} due to Postnikov, Reiner and Williams is clear. 
\end{remark}

Define the $\gamma$-polynomial of $A_n(t,q)$ and $\tilde{A}_{n}(t,q)$ by 
$$
\Gamma_{n}(y,q):=\sum_{k=0}^{\lfloor \frac{n-1}{2}\rfloor}\gamma_{n,k}(q)y^k\quad\text{and}\quad
\tilde{\Gamma}_n(y,q):=\sum_{k=0}^{\lfloor\frac{n}{2}\rfloor}\tilde{\gamma}_{n,k}(q)y^k,
$$
respectively. 
The following recurrence relation for  $\tilde{\Gamma}_n(y,q)$ follows directly from Theorem~\ref{rec:q-bino} and the relationships
$$\tilde{A_n}(t,q)=(1+t)^n\tilde{\Gamma}_n(t/(1+t)^2,q)\quad\text{and}\quad A_n(t,q)=(1+t)^{n-1}\Gamma_n(t/(1+t)^2,q).$$

\begin{corollary}\label{cor:rec:Gamma}
We have the following recursion for  $\tilde{\Gamma}_n(y,q)$:
\begin{equation}\label{rec:gamb}
\tilde{\Gamma}_{n+1}(y,q)=\tilde{\Gamma}_n(y,q)+y\sum_{k=1}^n{n\brack k}_qq^k\Gamma_k(y,q)\tilde{\Gamma}_{n-k}(y,q).
\end{equation}
\end{corollary}
\begin{remark}
One may also prove Theorem~\ref{q-bino:gam} by showing that the polynomials 
$$
\sum_{k=0}^{\lfloor\frac{n}{2}\rfloor}y^k\biggl(\sum_{\pi\in\tilde{\Gamma}_{n,k}}q^{\inv(\pi)}\biggr)
$$
satisfy the same recurrence relation as $\tilde{\Gamma}_{n}(y,q)$ in~\eqref{rec:gamb}.
\end{remark}

\section{Continued fractions and the unimodality of $\tilde{A}_n(t,-1)$}
\label{sec:unimodal}

In this section, we present a proof of the unimodality of $\tilde A_n(t,-1)$
and give a new $(p,q)$-extension of the $\gamma$-positivity of $\tilde{A}_n(t)$, via the machine of continued fraction.

\subsection{The unimodality of $\tilde A_n(t,-1)$} 
\label{sec:4.1}
Since the product of two palindromic and unimodal polynomials is again palindromic and unimodal (cf.~\cite{SWZ15}), recursion~\eqref{bino:odd} implies that Theorem~\ref{thm:bino} needs to be shown for even integers~$n$ only, that is, to show that the palindromic polynomial
\begin{equation}\label{def:star-A}
A^*_m(t):=1+t\sum_{k=1}^m{m\choose k}(1+t)^kA_k(t)
\end{equation} is unimodal for any integer $m\geq1$. 

The polynomials  $A_n^*(t)$ can be named the {\em binomial-Eulerian polynomials of type B}, since $(1+t)^nA_n(t)$ are the flag descent  polynomials~\cite{AB} over the Coxeter group of type B, namely, 
$$
(1+t)^nA_n(t)=\sum_{\sigma\in \mathfrak{B}_n}t^{\fdes(\sigma)},
$$
where $\mathfrak{B}_n$ is the set of  signed permutations of $[n]$ and $\fdes(\sigma)$  is the number of flag descents of $\sigma$. In order to prove the unimodality of $A^*_n(t)$, we need some preparation.

\begin{Def}
For any  permutation $\sigma \in \S_n$, the  numbers  of {\em cycle peaks, cycle valley, cycle double rises, cycle double descents, fixed points} of $\sigma$ are defined, respectively, by  
\begin{align*}
\cpeak(\sigma)&:=|\{i\in[n]\colon \sigma ^{-1}(i)<i>\sigma (i)\}|,\\
\cvalley(\sigma)&:=|\{i\in[n]\colon \sigma ^{-1}(i)>i<\sigma (i)\}|,\\
\cdrise(\sigma)&:=|\{i\in[n]\colon \sigma ^{-1}(i)<i<\sigma (i)\}|,\\
\cdfall(\sigma)&:=|\{i\in[n]\colon \sigma ^{-1}(i)>i>\sigma (i)\}|,\\
\fix(\sigma)&:=|\{i\in[n]\colon\sigma(i)=i\}|.
\end{align*}  
\end{Def}

For instance, if the cycle form of $\sigma\in\S_7$ is $(1462)(3)(57)$,  then 
$\cpeak(\sigma)=2, \cvalley(\sigma)=2, \cdrise(\sigma)=1,\cdfall(\sigma)=1$ and $\fix(\sigma)=1$.
Define 
$$
Q_n(a,b,c,d,\alpha)=\sum_{\sigma\in \S_n} a^{\cvalley(\sigma)}b^{\cpeak(\sigma)}c^{\cdfall(\sigma)}
d^{\cdrise(\sigma)}\alpha^{\fix(\sigma)}.
$$
We recall the following result from Zeng\cite{Zeng93}.
\begin{lemma}[Zeng] We have 
\begin{align}\label{fix-carlitz-scoville}
\sum_{n\geq 0}Q_n(a,b,c,d,\alpha)\frac{x^n}{n!}=e^{\alpha x}\frac{\alpha_1-\alpha_2}
{\alpha_1 e^{\alpha_2x}-\alpha_2 e^{\alpha_1 x}},
\end{align}
where $\alpha_1\alpha_2=ab$ and $\alpha_1+\alpha_2=c+d$.
Moreover,
\begin{align}\label{fix-carlitz-scoville-cf}
\sum_{n=0}^\infty Q_n(a,b,c,d,\alpha)x^n=
\cfrac{1}{1-\gamma_0 x-\cfrac{\beta_1x^2}{1-\gamma_1x-\cfrac{\beta_2x^2}{1-\gamma_2x-\cdots}}}
\end{align}
where 
$\gamma_n=n(c+d)+\alpha$ and $\beta_n=n^2ab$.
\end{lemma}
Since $\exc(\sigma)=\cvalley(\sigma)+\cdrise(\sigma)$, 
we have 
$A_n(t)=\sum_{\sigma\in\S_n}t^{\exc(\sigma)}=Q_n(t,1,1,t,1)$ and the well-known formula
\begin{equation}\label{gf:eulerian}
\sum_{n\geq 0} A_n(t)\frac{x^n}{n!}=\frac{t-1}{t-e^{x(t-1)}}
\end{equation}
is a special case of~\eqref{fix-carlitz-scoville}.

Next we compute the exponential generating function of $A^*_n(t)$.

\begin{lemma} We have  
\begin{align}\label{gf2}
\sum_{n\geq 0} A_n^*(t)\frac{x^n}{n!}= \frac{(t-1)e^{t^2x}}{t-e^{(t^2-1)x}}.
\end{align}
\end{lemma}
\begin{proof} 
It follows from~\eqref{def:star-A} and~\eqref{gf:eulerian} that
\begin{align*}
\sum_{n\geq 0}  A^*_n(t)\frac{x^n}{n!}&=(1-t)e^x
+t\sum_{n\geq 0}\sum_{k=0}^n {n\choose k}(1+t)^kA_k(t)\frac{x^n}{n!}\\
&=(1-t)e^x+te^x \sum_{n\geq0} A_n(t)\frac{(1+t)^nx^n}{n!}\\
&=\frac{(t-1)e^{t^2x}}{t-e^{(t^2-1)x}},
\end{align*}
as desired.
\end{proof}

We also need the following result~\cite[p.~306]{GJ04}.
\begin{lemma}[Jacobi--Rogers formula]\label{lem:JR}  Let $J_n$ be the sequence of coefficients in
the expansion 
$$
\sum_{n\geq 0} J_n x^n=\cfrac{1}{1-b_0t-\cfrac{\lambda_1 x^2}{
1-b_1x-\cfrac{\lambda_2 x^2}{
1-b_2x-\cdots}}}.
$$ 
Then for $n\geq 1$, we have
\begin{align*}
J_n&=b_0^{n}+\sum_{h\geq 0}\sum_{n_0, \ldots, n_{h}\geq 1\atop m_0, \ldots, m_{h+1}\geq 0} (b_0^{m_0}\cdots b_{h+1}^{m_{h+1}})(\lambda_1^{n_{0}}\cdots 
\lambda_{h+1}^{n_{h}}) \cdot \rho({\mathbf n}, \mathbf{m}),
\end{align*}
where 
$2{(n_0+\cdots +n_{h})}+{(m_0+\cdots +m_{h+1})}=n$ and 
$$
\rho({\mathbf n}, \mathbf{m}):=\prod_{j=0}^{h+1}{n_{j}+n_{j-1}-1\choose n_{j-1}-1}
{m_j+n_j+n_{j-1}-1\choose m_j}
$$
with the convention $n_{-1}=1$ and $n_{h+1}=0$.
\end{lemma} 

Now we are ready to prove Theorem~\ref{thm:bino}. 
\begin{proof}[{\bf Proof of Theorem~\ref{thm:bino}}]
By the discussions at the beginning of Section~\ref{sec:4.1}, we only need to show that $A^*_n(t)$ is unimodal for each $n\geq1$. 

Comparing the generating functions~\eqref{gf2} and~\eqref{fix-carlitz-scoville} we see that
$$
A^*_n(t)=(1+t)^n Q_n\left(t,1,1,t, \frac{t^2+t+1}{1+t}\right).
$$
It follows from~\eqref{fix-carlitz-scoville-cf} that 
\begin{align}\label{frac:star}
\sum_{n\geq 0} A^*_n(t)x^n=\cfrac{1}{1-b_0t-\cfrac{\lambda_1 x^2}{
1-b_1x-\cfrac{\lambda_2 x^2}{
1-b_2x-\cdots}}},
\end{align}
where  
$b_n=n(t+1)^2+1+t+t^2$ and $\lambda_n=n^2t(1+t)^2$.
In view of the Jacobi--Rogers formula (Lemma~\ref{lem:JR}), we have 
\begin{align}\label{sum:sigbi}
A^*_n(t)=b_0^{n}+\sum_{h\geq 0}
\sum_{n_0, \ldots, n_{h}\geq 1\atop m_0, \ldots, m_{h+1}\geq 0} (b_0^{m_0}\cdots b_{h+1}^{m_{h+1}})
(\lambda_1^{n_{0}}\cdots \lambda_{h+1}^{n_{h}})
 \cdot \rho({\mathbf n}, \mathbf{m}),
\end{align}
where 
\begin{equation}\label{center}
2{(n_0+\cdots +n_{h})}+{(m_0+\cdots +m_{h+1})}=n.
\end{equation}
Note that both the polynomials
 $b_n=(n+1)+(2n+1)t+(n+1)t^2$ and
  $\lambda_n=n^2t(1+t)^2$ are palindromic and unimodal. In view of~\eqref{center}, each product in the summation~\eqref{sum:sigbi} of $A^*_n(t)$ is also palindromic and unimodal with center of symmetry $n$. Hence $A^*_n(t)$ is palindromic and unimodal with center of symmetry $n$.
\end{proof}

\subsection{The log-convexity of $\tilde{A}_n(t)$ and $A^*_n(t)$}

There has been recent interest in the log-convexity of combinatorial sequences or polynomials (cf.~\cite{lw,zhu}). Let  $\mathcal{L}$ be the operator which maps a sequence $\{f_n(q)\}_{n\geq0}$ of polynomials  with real coefficients to the polynomial sequence $\{g_n(q)\}_{n\geq0}$ defined by
$$
g_i(q):=f_{i+1}(q)f_{i-1}(q)-f_i(q)^2.
$$
Then the sequence $\{f_n(q)\}_{n\geq0}$ is called {\em $k$-$q$-log-convex} if $\mathcal{L}^k\{f_n(q)\}_{n\geq0}$ is a sequence of polynomials with non-negative coefficients. 

Before we proceed to show the log-convexity of $\tilde{A}_n(t)$ and $A^*_n(t)$, we need the following continued fraction expansion for the ordinary generating function of $\tilde A_n(t)$ .

\begin{lemma} 
We have 
\begin{equation}\label{frac:bino}
\sum_{n\geq0} \tilde A_n(t) x^n=
\cfrac{1}{1-\gamma_0 x-\cfrac{\beta_1x^2}{1-\gamma_1x-\cfrac{\beta_2x^2}{1-\gamma_2x-\cdots}}},
\end{equation}
where $\gamma_n=(n+1)(t+1)$ and  $\beta_n=n^2 t$.
\end{lemma}
\begin{proof}
By~\eqref{bino:euler} we have 
$$
\sum_{n\geq 0} \tilde A_n(t)\frac{x^n}{n!}=\frac{(t-1)e^{tx}}{t-e^{(t-1)x}}.
$$
Comparing with \eqref{fix-carlitz-scoville}, we deduce that  $\tilde A_n(t)=Q_n(t,1,1, t, t+1)$. 
 The continued fraction expansion~\eqref{frac:bino} then follows from \eqref{fix-carlitz-scoville-cf}.
\end{proof}

\begin{theorem}
The polynomial sequences 
$\{\tilde{A}_n(q)\}_{n\geq1}$ and $\{A^*_n(q)\}_{n\geq1}$ are $3$-$q$-log-convex. 
\end{theorem}
\begin{proof}
By  a criterion of Zhu~\cite[Theorem~2.2]{zhu},  it is routine to check (for instance, by Maple) that the continued fraction expansion~\eqref{frac:star} implies the $3$-$q$-log-convexity of $\{A^*_n(q)\}_{n\geq1}$. The $3$-$q$-log-convexity of the sequence $\{\tilde{A}_n(q)\}_{n\geq1}$ follows in the same fashion from the continued fraction expansion~\eqref{frac:bino} for $\sum_{n\geq0} \tilde A_n(t) x^n$.
\end{proof}


\subsection{A new $(p,q)$-extension of the $\gamma$-positivity of $\tilde{A}_n(t)$ via continued fraction}
Let us introduce the polynomials $\hat{A}_n(t,p,q)$ by 
\begin{align}\label{new:bino}
\sum_{n\geq 0} \hat{A}_n(t,p,q)x^n =\cfrac{1}{1-b_0x-\cfrac{\lambda_1 x^2}{1-b_1x-\cfrac{\lambda_2 x^2}{1-\cdots}}},
\end{align}
where $b_n=(1+t)[n+1]_{p,q}$ and 
$\lambda_n=tq[n]_{p,q}^2$ with the usual notation $[n]_{p,q}=p^{n-1}+p^{n-2}q+\cdots+pq^{n-2}+q^{n-1}$.
In view of~\eqref{frac:bino}, we have 
$$
\hat{A}_n(t,1,1)=\tilde{A}_n(t),
$$ 
so $\hat{A}_n(t,p,q)$ is a $(p,q)$-analog of the binomial-Eulerian polynomials. The first few values of $\hat{A}_n(t,p,q)$ are 
\begin{align*}
\hat{A}_1(t,p,q)&=1+t,\\
\hat{A}_2(t,p,q)&=1+(2+q)t+t^2,\\
\hat{A}_3(t,p,q)&=1+(3+2q+q^2+pq)t+(3+2q+q^2+pq)t^2+t^3.
\end{align*}
%
The rest of this section is devoted to proving  a $(p,q)$-$\gamma$-positivity decomposition of $\hat{A}_n(t,p,q)$, involving crossings, nestings and generalized patterns of permutations.

\begin{Def}
For any  permutation $\sigma \in \S_n$, the numbers of {\em crossings, nestings, drops, 2--31 patterns, 31--2 patterns and foremaxima} are defined, respectively, by 
\begin{align*}
\cros(\sigma)&:=|\{(i,j)\in[n]\times[n]\colon i<j\leq\sigma_i<\sigma_j\text{ or }i>j>\sigma_i>\sigma_j\}|,\\
\nest(\sigma)&:=|\{(i,j)\in[n]\times[n]\colon i<j\leq\sigma_j<\sigma_i\text{ or }i>j>\sigma_j>\sigma_i\}|,\\
\drop(\sigma)&:=|\{2\leq i\leq n\colon \sigma_i<i\}|,\\
\text{(2--31)}\sigma&:=|\{(i,j)\colon 1\leq i<j\leq n-1\text{ and }\sigma_{j+1}<\sigma_i<\sigma_j\}|,\\
\text{(31--2)}\sigma&:=|\{(i,j)\colon 2\leq i<j\leq n\text{ and }\sigma_{i}<\sigma_{j}<\sigma_{i-1}\}|,\\
\fmax(\sigma)&:=|\{i\in[n]\colon  \sigma_j<\sigma_i\text{ for all $1\leq j<i$ and }\sigma_i<\sigma_{i+1}\}|.
\end{align*}
For instance, if $\sigma=42513\in\S_5$, then $\cros(\sigma)=1$, $\nest(\sigma)=1$, $\drop(\sigma)=2$, $\text{(2--31)}\sigma=2$, $\text{(31--2)}\sigma=2$ and $\fmax(\sigma)=0$.
\end{Def}

The $q$-binomial-Eulerian polynomial $\hat{A}_n(t,1,q)$ arose in Williams' enumeration of totally positive Grassmann cells (see~\cite[Lemma~5]{wi}), while the $(p,q)$-analog $\hat{A}_n(t,p,q)$ first appeared in the work of Corteel~\cite[Proposition~7]{cor}, where she showed that 
\begin{equation}\label{int:corteel}
\hat{A}_n(t,p,q)=\sum_{\sigma\in \S_n} 
p^{\nest(\sigma)}q^{\cros(\sigma)+\drop(\sigma)}
(1+t)^{\fix(\sigma)}t^{\exc(\sigma)}.
\end{equation}
Consider the common enumerative polynomial (see~\cite[Theorem~5]{sz})
\begin{align}
 B_n(p,q,t,u,v, w, y)&=
\sum_{\sigma\in \S_n} 
p^{\nest(\sigma)}q^{\cros(\sigma)}t^{\drop(\sigma)}
u^{\cdrise(\sigma)}v^{\cdfall(\sigma)}w^{\cvalley(\sigma)} y^{\fix(\sigma)}\nonumber\\
&=\sum_{\sigma\in \S_n} 
p^{(2-31) \sigma}q^{(31-2) \sigma}t^{\des(\sigma)}
u^{\da^*(\sigma)-\fmax(\sigma)}v^{\dd(\sigma)}w^{\valley^*(\sigma)} y^{\fmax(\sigma)},\label{line-stat}
\end{align}
where $\da^*(\sigma)=\da(\sigma)+\chi(\sigma_1<\sigma_2)$ and $\valley^*(\sigma)=\valley(\sigma)-\chi(\sigma_1<\sigma_2)$. 
Since $\cdrise(\sigma)+\cvalley(\sigma)=\exc(\sigma)$, it follows from~\eqref{int:corteel} that 
\begin{equation*}
\hat{A}_n(t,p,q)=B_n(p,q,q,t,1,t,1+t). 
\end{equation*}
This relationship together with interpretation~\eqref{line-stat} of $B_n(p,q,t,u,v, w, y)$ gives another interpretation for $\hat{A}_n(t,p,q)$:
\begin{equation}
\hat{A}_n(t,p,q)=
\sum_{\sigma\in \S_n} 
p^{(2-31) \sigma}q^{(31-2) \sigma+\des(\sigma)}
(1+t)^{\fmax(\sigma)}t^{\des(\sigma)},
\end{equation}
in view of the symmetry $\hat{A}_n(t,p,q)=t^n\hat{A}_n(t^{-1},p,q)$ proved below.  
\begin{theorem}
We have 
\begin{equation*}
\hat{A}_n(t,p,q)=\sum_{k=0}^{\lfloor\frac{n}{2}\rfloor}\hat{\gamma}_k(p,q)t^k(1+t)^{n-2k},
\end{equation*}
where  
\begin{equation}\label{gam:hat}
\hat{\gamma}_k(p,q)=\sum_{\sigma\in\hat{\Gamma}_{n,k}}p^{\nest(\sigma)}q^{\cros(\sigma)+k}=\sum_{\sigma\in\tilde{\Gamma}_{n,k}}p^{(2-31) \sigma}q^{(31-2) \sigma+k}
\end{equation}
with  $\hat{\Gamma}_{n,k}:=\{\sigma\in\S_n\colon \cdfall(\sigma)=0,\drop(\sigma)=k\}$.
\end{theorem}
\begin{proof}
Shin and the third author proved in~\cite[Eq.~(34)]{sz} the following continued fraction expansion:  
\begin{align}\label{sz:con}
\sum_{n\geq 0} B_n(p,q,t,u,v,w,y)x^n =\cfrac{1}{1-b_0x-\cfrac{\lambda_1 x^2}{1-b_1x-\cfrac{\lambda_2 x^2}{1-\cdots}}}
\end{align}
with
$b_n=yp^n+(qu+tv)[n]_{p,q}$ and 
$\lambda_n=tw[n]_{p,q}^2$. Let 
$$
P_n(p,q,y):=B_n(p,q,q,1,0,y,1).
$$
It follows from~\eqref{sz:con} that 
\begin{align}\label{gam:drop}
\sum_{n\geq 0} P_n(p,q,y)x^n =\cfrac{1}{1-b_0x-\cfrac{\lambda_1 x^2}{1-b_1x-\cfrac{\lambda_2 x^2}{1-\cdots}}}
\end{align}
with $b_n=[n+1]_{p,q}$ and
$\lambda_n=yq[n]_{p,q}^2$. Comparing~\eqref{gam:drop} with~\eqref{new:bino} yields 
$$
\hat{A}_n(t,p,q)=(1+t)^nP_n(p,q,y),
$$
where $y=\frac{t}{(1+t)^2}$. This is equivalent to 
\begin{equation}\label{eq:gam}
B_n(p,q,q,1,0,y,1)=\sum_{k=0}^{\lfloor\frac{n}{2}\rfloor}\hat{\gamma}_{n,k}(p,q)y^k.
\end{equation}
Since $\cvalley(\sigma)=\drop(\sigma)$ (resp.~$\valley^*(\sigma)=\des(\sigma)$) whenever $\cdfall(\sigma)=0$ (resp.~$\dd(\sigma)=0$), the interpretations of $\hat{\gamma}_{n,k}(p,q)$ in~\eqref{gam:hat} then follow from~\eqref{eq:gam} and the definition of $B_n(p,q,t,u,v, w, y)$.
\end{proof}

\begin{remark}
Foata's first fundamental transformation (cf.~\cite[Prop.~1.3.1]{st0}) establishes a one-to-one correspondence between  $\tilde{\Gamma}_{n,k}$ and $\hat{\Gamma}_{n,k}$.
\end{remark}

\section{Closing remarks}
The elementary approach via quadratic  recursion in Section~\ref{sign:qbino} could be applied to prove other known or new sign-balance identities for the Eulerian distributions on restricted permutations, including the descent polynomials of Andr\'e or Simsun permutations and the excedance polynomials of $321$-avoiding permutations. The interested reader is referred to an extended version~\cite{lwz} of this paper for details.

Note that Wachs~\cite{wa}
used a combinatorial involution 
to prove \eqref{foata}. It would be interesting to find 
analogous 
combinatorial proof for Theorem~\ref{sign:qeul} and Corollary~\ref{sign-qbino}. 
A combinatorial polynomial $h(t)=\sum_{k=0}^na_k(q)t^k\in\mathbb{N}[q][t]$ is {\em $q$-log-concave} if 
$a_k(q)^2-a_{k-1}(q)a_{k+1}(q)\in \mathbb{N}[q]$.
We propose the following conjectures.
\begin{conjecture}
The $q$-binomial-Eulerian polynomial $\tilde{A}_n(t,q)$ is $q$-log-concave for $n\geq1$.
\end{conjecture}

\begin{conjecture}\label{conj:log}
The signed binomial-Eulerian polynomial $\tilde{A}_n(t,-1)$ is log-concave for $n\geq1$.
\end{conjecture}

The validation  of  Conjecture~\ref{conj:log} would imply  Theorem~\ref{thm:bino}.

\section*{Acknowledgments}

We thank the anonymous referee for the 
insightful comments and suggestions.

The first author's research was supported by the National Science Foundation of China grants 11871247 and 11501244, and by the Training Program Foundation for Distinguished Young Research Talents of Fujian Higher Education. The second author was partially supported by the National  Science Foundation of China grant 11671037. Part of this work was done while the first and third authors were visiting Institute for Advanced study in Mathematics of Harbin Institute of Technology (HIT) in the summer of 2018.

 \end{document}